\documentclass[10pt]{amsart}
\usepackage{amsfonts,amssymb,amscd,amsmath,enumerate,verbatim,calc}
\numberwithin{equation}{section}
\newtheorem{thm}{Theorem}[section]
\newtheorem{cor}[thm]{Corollary}
\newtheorem{lem}[thm]{Lemma}
\newtheorem{prop}[thm]{Proposition}
\newtheorem{defn}[thm]{Definition}

\newtheorem{rem}[thm]{Remark}

\DeclareMathOperator{\Tor}{Tor}
\DeclareMathOperator{\Ext}{Ext}
\DeclareMathOperator{\Supp}{Supp}
\DeclareMathOperator{\V}{V}
\DeclareMathOperator{\Hom}{Hom}
\DeclareMathOperator{\Ker}{Ker}
\DeclareMathOperator{\Coker}{Coker}
\DeclareMathOperator{\Image}{Im}

\DeclareMathOperator{\cd}{cd}
\DeclareMathOperator{\q}{q}
\DeclareMathOperator{\Ass}{Ass}

\DeclareMathOperator{\Min}{Min}

\DeclareMathOperator{\lc}{H}

\DeclareMathOperator{\Spec}{Spec}
\DeclareMathOperator{\G}{\Gamma}

\DeclareMathOperator{\pd}{pd}

\newcommand{\lo}{\longrightarrow}

\newcommand{\fa}{\mathfrak{a}}
\newcommand{\fb}{\mathfrak{b}}

\newcommand{\fp}{\mathfrak{p}}
\newcommand{\fq}{\mathfrak{q}}

\begin{document}

\title[Upper bounds of generalized local cohomology modules]
 {Upper bounds for finiteness of generalized local cohomology modules}

\bibliographystyle{amsplain}

    \author[M. Aghapournahr]{M. Aghapournahr}
\address{ Department of Mathematic, Faculty of Science,  Arak University, Arak, 38156-8-8349, Iran.}
\email{m-aghapour@araku.ac.ir}


\keywords{Generalized local cohomology module, Serre subcategory,
cohomological dimension.}

\subjclass[2000]{13D45, 13D07}


\begin{abstract}
Let $R$ be a commutative Noetherian ring with non-zero identity and
$\fa$ an ideal of $R$. Let $M$ be a finite $R$--module of
of finite projective dimension and $N$ an arbitrary finite $R$--module.
We characterize the membership of the generalized local cohomology modules
$\lc^{i}_{\fa}(M,N)$ in certain Serre
subcategories of the category of modules from upper bounds. We define and study
the properties of a generalization of cohomological dimension of generalized local
cohomology  modules. Let $\mathcal S$ be a Serre subcategory of the category
of $R$--modules and $n \geqslant \pd M$ be an integer such that $\lc^{i}_{\fa}(M,N)$
 belongs to $\mathcal S$ for all $i> n$. If $\fb$ is an ideal of $R$
 such that $\lc^{n}_{\fa}(M,N/{\fb}N)$ belongs to $\mathcal S$, It is also
 shown that  the module
  $\lc^{n}_{\fa}(M,N)/{\fb}\lc^{n}_{\fa}(M,N)$ belongs to $\mathcal S$.
  \end{abstract}

\maketitle

\section{Introduction}
Throughout this paper $R$ is a commutative noetherian ring. Let  $\fa$ be
an ideal of $R$, $M$ be a finite $R$--module of
of finite projective dimension and $N$ an arbitrary finite $R$--module. The
notion of generalized local
cohomology was introduced by J. Herzog \cite{He}. The $i$--th generalized
local cohomology modules of $M$ and $N$ with respectt to $\fa$ is defined
 by
\begin{center}
$\lc^{i}_{\fa}(M,N) \cong \underset{n}\varinjlim
\Ext^{i}_{R}(M/{\fa}^{n}M,N).$
\end{center}

 It is clear that
$\lc^{i}_{\fa}(R,N)$ is just the the ordinary local cohomology module $\lc^{i}_{\fa}(N)$.
This concept was studied in the articles \cite{S}, \cite{BZ} and
\cite{Ya}.

For ordinary local cohomology module there is the important concept {\it{cohomological dimension}}
 of an $R$--module $N$ with respect to an ideal $\fa$ of $R$. It is denoted by
\begin{center}
$\cd_{\fa}(N)= \sup\{i\geqslant 0 | \lc^{i}_{\fa}(N)\neq 0\}$
\end{center}
This notion has been studied by several authors; see, for example \cite{F}, \cite{Ha}, \cite{O}, \cite{HL} and
\cite{DNT}.

Hartshorn \cite{Ha} has defined the notion   $\q_{\fa}(R)$ as the greatest integer $i$ such that  $\lc^{i}_{\fa}(R)$
 is not Artinian. Dibaei and yassemi  \cite{DY} extended this notion to arbitrary finite $R$--modules as

\begin{center}
 $\q_{\fa}(N)=\sup \{i\geqslant 0 | \lc^{i}_{\fa}(N)~~\text{is not Artinian}\}$
\end{center}

 Recall that a subclass of the class of all modules is called Serre class, if it is closed under taking submodules,
 quotients and extensions. Examples are given by the class of finite modules, Artinian modules and etc.
   In  \cite[Theorem 3.1 and 3.3]{AM} the Author and
  Melkersson characterized the membership of ordinary local cohomology modules in certain
  Serre class of the class of modules from upper bounds they also introduced
{\it Serre cohomological dimension of a module with respect to an ideal} \cite[Definition 3.5]{AM} as

\begin{center}
${\cd}_{(\fa,\mathcal S)}(N)=\sup\{n\geq 0|\lc^{i}_{\fa}(N)
\text{ is not in } \mathcal S \}$.
\end{center}

see also \cite[Definition 3.4]{AT}. Note that when $\mathcal S=\{0\}$ then ${\cd}_{(\fa,\mathcal S)}(N)={\cd}_{\fa}(N)$
 and when $\mathcal S$ is the class of Artinian modules, then ${\cd}_{(\fa,\mathcal S)}(N)={\q}_{\fa}(N)$.

Amjadi and Naghipour in \cite{AN} (resp. Asgharzadeh, Divaani-Aazar and Tousi in \cite{DNT})
extended ${\cd}_{\fa}(N)$ (resp. ${\q}_{\fa}(N)$) to generalized local cohomology modules as

\begin{center}
$\cd_{\fa}(M,N)=\sup \{i\geqslant 0 | \lc^{i}_{\fa}(M,N)\neq 0\}$
\end{center}
\begin{center}
  $(\text{resp}.~\q_{\fa}(M,N)=\sup \{i\geqslant 0 | \lc^{i}_{\fa}(M,N)~~\text{is not Artinian}\}).$
\end{center}
They also proved basic results about related notions. Also there are some other attempts to study generalized
local cohomology modules from upper bounds, see \cite[Corollary 2.7]{CT} and \cite[Theorem 5.1, Lemma 5.2 and Corollary 5.3 ]{CH}.

Our objective in this paper is to characterize the membership of generalized local cohomology
modules in certain Serre class of the category of $R$--modules from upper bounds. We will do it in section 2.
 Our main results in this section are theorems \ref{2-1}, \ref{T:casen+1} and \ref{P:loc/aloc}. In section 3, we will
define and study the {\it Serre cohomological dimension of two modules with respect to an ideal}. Our definition and
results in this paper improve and generalize all of the above mentioned one.
For unexplained terminology we refer to \cite{BSh} and \cite{BH}.

\section{main results}





The following theorem characterize the membership of generalized local cohomology modules to
a certain Serre class from upper bounds.

\begin{thm} \label {2-1}
Let $\mathcal S$ be a Serre subcategory of the category of
$R$--modules. Let $\fa$ an ideal of $R$, $M$ be a finite $R$--module of finite
projective dimension and $N$ an arbitrary finite $R$--module. Let $n \geqslant\pd M$ be a
non-negative integer. Then the following statements are equivalent:
     \begin{itemize}
               \item[(i)] $\lc^{i}_{\fa}(M,N)$ is in $\mathcal S$ for all $i> n$.
               \item[(ii)] $\lc^{i}_{\fa}(M,L)$ is in $\mathcal S$ for all $i> n$ and for
                               every finite $R$--module $L$ such that $\Supp_R(L)\subset\Supp_R(N)$.
               \item[(iii)] $\lc^{i}_{\fa}(M,R/\fp)$ is in $\mathcal S$ for all
                               $\fp\in\Supp_R(N)$ and all $i> n$.
               \item[(iv)] $\lc^{i}_{\fa}(M,R/\fp)$ is in $\mathcal S$ for all
                               $\fp\in\Min\Ass_R(N)$ and all $i> n$.
      \end{itemize}
\end{thm}

\begin{proof}
We use descending induction on $n$. So we may assume that all
conditions are equivalent when $n$ is replaced by $n+1$ using \cite[
Theorem 2.5]{Ya}.

(i)$\Rightarrow$(iii). We want to show that
$\lc^{n+1}_{\fa}(M,R/\fp)$ is in $\mathcal S$ for each
$\fp\in\Supp_R(N)$. Suppose the contrary and let $\fp\in\Supp_R(N)$
be maximal of those $\fp\in\Supp_R(N)$ such that
$\lc^{n+1}_{\fa}(M,R/\fp)$ is not in $\mathcal S$. Since
$\fp\in\Supp_R(N)$, there is by \cite[Chap.(ii), \S\  4, $n^o$ 4,
Proposition 20]{Bo} a nonzero map $f:M\lo R/\fp$. Let
$\fb\supsetneqq \fp$ be the ideal of $R$ such that $\Image f=
\fb/\fp$. The exact sequence $0\rightarrow \Ker{f}\rightarrow
M\rightarrow \Image{f}\rightarrow 0$, yields the exact sequence
$$
\lc^{n+1}_{\fa}(M,N)\lo \lc^{n+1}_{\fa}(M,\Image f)\lo
\lc^{n+2}_{\fa}(M,\Ker f).
$$
Since $\Supp_R(\Ker f)\subset\Supp_R(N)$, by induction
$\lc^{n+2}_{\fa}(M,\Ker f)$ belongs to $\mathcal S$. It follows that
$\lc^{n+1}_{\fa}(M,\Image f)$ belongs to $\mathcal S$. There is a
filtration
$$
0=N_t\subset N_{t-1}\subset N_{t-2}\subset \dots\subset N_0=R/\fb
$$
of submodules of $R/\fb$, such that for each $0\leqslant i\leqslant
t$, $N_{i-1}/N_i\cong R/{\fq}_i$ where ${\fq}_i\in \V(\fb)$. Then by
the maximality of $\fp$, $\lc^{n+1}_{\fa}(M,R/{\fq}_i)$ is in
$\mathcal S$. Use the exact sequences $0\rightarrow N_{i}\rightarrow
N_{i-1}\rightarrow R/{\fq}_i\rightarrow 0$, to conclude that
$\lc^{n+1}_{\fa}(M,R/\fb)$ is in $\mathcal S$. Next the exact
sequence $0\rightarrow \Image f\rightarrow R/\fp\rightarrow
R/{\fb}\rightarrow 0$, yields the exact sequence
$$
\lc^{n+1}_{\fa}(M,\Image f)\lo \lc^{n+1}_{\fa}(M,R/\fp)\lo
\lc^{n+1}_{\fa}(M,R/\fb).
$$
It follows that $\lc^{n+1}_{\fa}(M,R/\fp)$ is in $\mathcal S$ which
is a contradiction.

(iii)$\Rightarrow$(ii). Use a filtration for $N$ as above.

(iv)$\Rightarrow$(iii). Let $\fp\in\Supp_R(N)$. Then $\fp\supset
\fq$ for some $\fq\in\Min\Ass_R(N)$. Hence $\fp\in\Supp_R(R/\fq)$.
Applying (i)$\Rightarrow$ (iii), it follows that
$\lc^{i}_{\fa}(M,R/\fp)$ is in $\mathcal S$ for all $i> n$.
\end{proof}

\begin{cor}\label{C:supN=supM}
Let $\mathcal S$ be a Serre subcategory of the category of
$R$--modules. Let $\fa$ an ideal of $R$ and $M$ be a finite $R$--module of finite
projective dimension. Let $n \geqslant\pd M$ be a
non-negetive integer. If $L$ and $N$ are
finite $R$--modules such that $\Supp_R(L)=\Supp_R(N)$, then
$\lc^{i}_{\fa}(M,L)$ is in $\mathcal S$ for all $i> n$ if and only
if $\lc^{i}_{\fa}(M,N)$ is in $\mathcal S$ for all $i> n$.
\end{cor}

\begin{defn} \label {2-4}  {\rm(}see \cite [Definition 2.1]{AM} and \cite [Definition 3.1]{ATV}{\rm)}
Let $\mathcal{M}$ be a Serre subcategory of the category of
$R$--modules. We say that $\mathcal{M}$ is a {\it Melkersson
subcategory with respect to the ideal $\fa$} if for any
$\fa$--torsion $R$--module $X$, $0:_{X}\fa$ is in $\mathcal{M}$
implies that $X$ is in $\mathcal{M}$. $\mathcal{M}$ is called {\it
Melkersson subcategory} when it is a Melkersson subcategory with
respect to all ideals of $R$.
\end{defn}

When $\mathcal M$ is Melkersson subcategory of the category of
$R$--modules, we are able to weaken the condition $(iii)$ in
\ref{2-1}
 to require that
$\lc^{i}_{\fa}(M,R/\fp)$ is in $\mathcal M$ for all
$\fp\in\Supp_R(N)$, just for $i=n+1$.

\begin{thm}\label{T:casen+1}
Let $\mathcal M$ is Melkersson subcategory of the category of
$R$--modules $R$--module. Let $\fa$ an ideal of $R$ and $M$ be a finite $R$--module of finite
projective dimension. Let $n \geqslant\pd M$ be a
non-negetive integer. Then for
each finite $R$--module $N$ the conditions in theorem \ref{2-1} are
equivalent to:
\begin{enumerate}
  \item[(v)] $\lc^{n+1}_{\fa}(M,R/\fp)$ is in
    $\mathcal M$ for all $\fp\in\Supp_R(N)$.
\end{enumerate}
\end{thm}
\begin{proof}
(v)$\Rightarrow$(iv). We prove by induction on $i\geq n+2$ that
$\lc^{i}_{\fa}(M,R/{\fp})$ is in $\mathcal M$ for all
$\fp\in\Supp_R(N)$. It is enough to treat the case $i=n+2$. Suppose
that $\lc^{n+2}_{\fa}(M,R/{\fp})$ is not in $\mathcal M$ for some
$\fp\in \Supp_R(M)$. It follows that $\fa\not\subset \fp$, since
otherwise $\lc_{\fa}^{n+2}(M,R/{\fp})=0$, because $n+2>0$. Take
$x\in\fa\setminus \fp$ and put $L=R/({\fp}+{x}R)$. Then
$\Supp_R(L)\subset\Supp_R(N)$. We have a finite filtration
$$
0=L_t\subset L_{t-1}\subset L_{t-2}\subset \dots\subset L_0=L
$$
such that $L_{i-1}/L_{i}\cong R/{\fp}_i$ for each $1 \leq i\leq t$
where ${\fp}_i\in \Supp_R(N)$. Using the exact sequence
$$
\lc^{n+1}_{\fa}(M,L_i)\lo \lc^{n+1}_{\fa}(M,L_{i-1})\lo
\lc^{n+1}_{\fa}(M,R/{\fp}_i)
$$
for each $1 \leq i\leq t$, shows that $\lc^{n+1}_{\fa}(M,L)$ is in
$\mathcal M$. Consider the exact sequence $0\rightarrow
R/\fp\overset x \rightarrow R/\fp\rightarrow L\rightarrow 0,$ which
induces the following exact sequence
$$
\lc^{n+1}_{\fa}(M,L)\lo \lc^{n+2}_{\fa}(M,R/{\fp}) \overset x\lo
\lc^{n+2}_{\fa}(M,R/{\fp}).
$$
This shows that $0:_{\lc^{n+2}_{\fa}(M,R/{\fp})}{x}$ is in $\mathcal
M$. Since $\lc^{n+2}_{\fa}(M,R/{\fp})$ is $\fa$--torsion, by
\cite[Lemma 2.3]{AM} $\lc^{n+2}_{\fa}(M,R/{\fp})$ is in $\mathcal M$, which
is a contradiction.
\end{proof}


\begin{rem}\label{R:appl}
In theorems \ref{2-1} and \ref{T:casen+1} we may specialize
$\mathcal S$ to any of the Melkersson subcategories, given in \cite[Example 2.4]{AM}
to obtain characterizations of artinianness, vanishing, finiteness
of the support etc.
\end{rem}

Cho and Tang gave some parts of them in \cite[Theorem 2.5,
2.6 and Corrolary 2.7]{CT}  for the case of artinianness in local rings. In the case of vanishing,
i.e., when $\mathcal S$ merely consists of zero modules and finiteness of
 support (in local case) some parts
of them were studied in \cite[Theorem B]{AN} and \cite[Theorem 5.1 part (a)]{CH}.
 These authors used Gruson's theorem, \cite[Theorem
4.1]{V}, while we just used the maximal condition in a noetherian
ring.


\begin{thm}\label{P:loc/aloc}
Let $\mathcal S$ be a Serre subcategory of the category of
$R$--modules and $M$ be a finite $R$--module of finite projective
dimension. Let $N$ be a finite $R$--module, $\fa$
 be an ideal of $R$ and $n \geqslant \pd M$ be an integer such that $\lc^{i}_{\fa}(M,N)$
 belongs to $\mathcal S$ for all $i> n$. If $\fb$ is an ideal of $R$
 such that $\lc^{n}_{\fa}(M,N/{\fb}N)$ belongs to $\mathcal S$, then the module
  $\lc^{n}_{\fa}(M,N)/{\fb}\lc^{n}_{\fa}(M,N)$ belongs to $\mathcal S$.
\end{thm}
\begin{proof}
Suppose $\lc^{n}_{\fa}(M,N)/{\fb}\lc^{n}_{\fa}(M,N)$ is not in
$\mathcal S$. Let $L$ be a maximal submodule of $N$ such that
$\lc^{n}_{\fa}(M,N/L)\otimes_R{R/\fb}$ is not in $\mathcal S$. Let
$T\supset L$ be such that $\G_{\fb}(N/L)=T/L$. Since
$\Supp_R(T/L)\subset{\V(\fb)\cap \Supp_R(N)}$,\
$\lc^{i}_{\fa}(M,T/L)$ belongs to $\mathcal S$ for all $i\geq n$ by
\ref{2-1}.

From the exact sequence $0 \rightarrow T/L\rightarrow N/L\rightarrow
N/T\rightarrow 0$, we get the exact sequence
\begin{center}
$\lc^{n}_{\fa}(M,T/L)\lo \lc^{n}_{\fa}(M,N/L)\overset{f}\lo
\lc^{n}_{\fa}(M,N/T)\lo \lc^{n+1}_{\fa}(M,T/L)$.
\end{center}
$\Tor^R_{i}(R/\fb,\Ker f)$ and $\Tor^R_{i}(R/\fb,\Coker f)$ are in
$\mathcal S$ for all $i$, because $\Ker f$ and $\Coker f$ are in
$\mathcal S$. It follows from \cite[Lemma 3.1]{Mel}, that $\Ker
(f\otimes{R/\fb})$ and $\Coker (f\otimes{R/\fb})$ are in $\mathcal
S$. Since $\lc^{n}_{\fa}(M,N/L)\otimes_R{R/\fb}$ is not in $\mathcal
S$,
 the module $\lc^{n}_{\fa}(M,N/T)\otimes_R{R/\fb}$ can not be in $\mathcal S$. By the maximality of $L$, we get $T=L$.
  We have shown that $\G_{\fb}(N/L)=0$ and therefore we can take $x\in \fb$ such that the sequence
   $0\rightarrow N/L\overset{x}\rightarrow N/L\rightarrow N/(L+{x}N)\rightarrow 0$ is exact. Thus we get the exact
   sequence
\begin{center}
$\lc^{n}_{\fa}(M,N/L)\overset{x}\rightarrow \lc^{n}_{\fa}(M,N/L)\rightarrow
\lc^{n}_{\fa}(M,N/L+{x}N)\rightarrow \lc^{n+1}_{\fa}(M,N/L).$
\end{center}
 This yields the exact sequence
\begin{center}
$0\rightarrow
\lc^{n}_{\fa}(M,N/L)/{x}\lc^{n}_{\fa}(M,N/L)\rightarrow
\lc^{n}_{\fa}(M,N/L+{x}N)\rightarrow C\rightarrow 0$,
\end{center}
where $C\subset \lc^{n+1}_{\fa}(M,N/L)$ and thus $C$ is in $\mathcal
S$.

Note that $x\in \fb$. Hence we get the exact sequence
\begin{center}
$\Tor^R_{1}(R/\fb,C)\lo \lc^{n}_{\fa}(M,N/L)\otimes_R{R/\fb}\lo
\lc^{n}_{\fa}(M,N/(L+{x}N))\otimes_R{R/\fb}$
\end{center}
However $L\subsetneqq{(L+{x}N)}$ and therefore
$\lc^{n}_{\fa}(M,N/(L+{x}N))\otimes_R{R/\fb}$ belongs to $\mathcal
S$ by the maximality of $L$. Consequently
$$\lc^{n}_{\fa}(M,N/L)\otimes_R{R/\fb}$$ is in $\mathcal S$ which is a
contradiction.
\end{proof}

\begin{lem} \label {2-0}
Let $M$ and $N$ be two finite $R$--modules such that $\Supp_R(M)
\cap \Supp_R(N)\subseteq \V(\fa)$. Then $\lc^{i}_{\fa}(M, N)\cong
\Ext^i_{R}(M, N)$ for all $i\geq 0$.
\end{lem}

\begin{proof}
There is a minimal injective resolution $E^\bullet$ of $N$ such that
$\Supp_R(E^i)\subseteq \Supp_R(N)$ for all $i\geq 0$. Since
$\Supp_R(\Hom_{R}(M, E^i))\subseteq \Supp_R(M) \cap
\Supp_R(N)\subseteq \V(\fa)$, $\Hom_{R}(M, E^i)$ is
$\fa$--torsion. Therefore, for all $i\geq 0$,
$$\begin{array}{llll}
\lc^i_\fa(M, N)\!\!&= \ \ \lc^{i}(\G_{\fa}(\Hom_{R}(M,
E^\bullet)))\\&= \ \ \lc^{i}(\Hom_{R}(M, E^\bullet))\\&= \ \
\Ext^i_{R}(M, N),
\end{array}$$
as we desired.
\end{proof}


Asgharzadeh, Divaani-Aazar and Tousi, in \cite[Theorem 3.3 (i)]{ADT} proved
the following corollary when $\mathcal S$ is the category of Artinian $R$--modules
 with an strong assumption that $N$ has finite Krull dimension. This condition is
 very near to local case, while it is a simple conclusion of Theorem \ref{P:loc/aloc}
 without that strong assumption.

\begin{cor}\label{T:H/bH}
Let $\mathcal S$ be a Serre subcategory of the category of
$R$--modules and $M$ be a finite $R$--module of finite projective
dimension. Let $N$ be a finite $R$--module, $\fa$ be an ideal of $R$
and $n>\pd M$ be an integer such that $\lc^{i}_{\fa}(M,N)$ belongs
to $\mathcal S$ for all $i>n$ , then
$\lc^{n}_{\fa}(M,N)/{\fa}\lc^{n}_{\fa}(M,N)$ belongs to $\mathcal
S$.
\end{cor}
\begin{proof}
Note that $\lc^{i}_{\fa}(M,N/{\fa}N)\cong \Ext^{i}_{R}(M,N/{\fa}N)$
for all $i\geqslant 0$ by lemma \ref{2-0}, so
$\lc^{n}_{\fa}(M,N/{\fa}N)=0$, now the proof is complete by theorem
\ref{P:loc/aloc} .
\end{proof}

\begin{cor}\label{T:H/aH}
Let $\mathcal S$ be a Serre subcategory of the category of
$R$--modules and $M$ be a finite $R$--module of finite projective
dimension. Let $N$ be a finite $R$--module, $\fa$ be an ideal of $R$
and $n>\pd M$ be an integer such that $\lc^{i}_{\fa}(M,N)$ belongs
to $\mathcal S$ for all $i>n$ , then $\lc^{n}_{\fa}(M,N)$ is not
finitely generated. In particular $\lc^{n}_{\fa}(M,N)\neq 0$.
\end{cor}
\begin{proof}
Suppose  $\lc^{n}_{\fa}(M,N)$ is finitely generated. Then there
exist an integer $t$ such that ${\fa}^t\lc^{n}_{\fa}(M,N)=0$ but
$\lc^{n}_{{\fa}^t}(M,N)\cong \lc^{n}_{\fa}(M,N)$ for all $i>0$. So
$\lc^{n}_{\fa}(M,N)\cong
\lc^{n}_{{\fa}^t}(M,N)/{\fa}^t\lc^{n}_{{\fa}^t}(M,N)$, is in
$\mathcal S$, which is a contradiction.
\end{proof}

\section{Cohomological dimension with respect to Serre class}
In the following we introduce the last integer such that the generalized
local cohomology modules belong to a Serre class.

\begin{defn}\label{D:t}
Let $\mathcal S$ be a Serre subcategory of the category of
$R$--modules. Let $\fa$ be an ideal of $R$ and $M,N$ two
$R$--modules. We define
\begin{center}
${\cd}_{(\fa,\mathcal S)}(M,N)=\sup\{n\geq 0|\lc^{i}_{\fa}(M,N)
\text{ is not in } \mathcal S \}$.
\end{center}
with the usual convention that the suprimum of the empty set of integers
 is interpreted as $-\infty$.

For example when $\mathcal S$=\{0\}, then ${\cd}_{(\fa,\mathcal
S)}(M,N)= \cd_{\fa}(M,N)$ and when $\mathcal S$ is the class of
Artinian modules, then ${\cd}_{(\fa,\mathcal S)}(M,N)=\q_{\fa}(M,N)$
as in \cite{AN} and \cite{ADT}.
\end{defn}

In the following we study the main properties of this invariant.

\begin{prop}\label{P:main}
Let $\mathcal S$ be a Serre subcategory of the category of
$R$--modules. Let $\fa$ be an ideal of $R$ and $M$ a finite
$R$--module of finite projective dimension. The following statements
hold.
\begin{enumerate}
  \item[\rm{(}a\rm{)}] Let $\mathcal S_1,\mathcal S_2$ be two
    Serre subcategories of the category of $R$--modules
     such that $\mathcal S_1\subset \mathcal S_2$.
     Then ${\cd}_{(\fa,\mathcal
S_2)}(M,N)\leq{\cd}_{(\fa,\mathcal S_1)}(M,N)$
     for every finite $R$--module $N$.
     In particular ${\cd}_{(\fa,\mathcal
S)}(M,N)\leq\cd_{\fa}(M,N)$ for each
     Serre subcategory $\mathcal S$ of the category of $R$--modules.
  \item[\rm{(}b\rm{)}] If $L$ and $N$ are finite $R$--modules s.t.
    $\Supp_R(L)\subset\Supp_R(N)$, then
    ${\cd}_{(\fa,\mathcal S)}(M,L)\leq{\cd}_{(\fa,\mathcal S)}(M,N)$
 and
    equality holds if $$\Supp_R(L)= \Supp_R(N).$$
  \item[\rm{(}c\rm{)}] Let $0\rightarrow N^{\prime}\rightarrow N
    \rightarrow N^{\prime\prime}\rightarrow 0$
    be an exact sequence of finite $R$--modules. Then
    $${\cd}_{(\fa,\mathcal S)}(M,N)=
    \max \{{\cd}_{(\fa,\mathcal S)}(M,N^{\prime}),
    {\cd}_{(\fa,\mathcal S)}(M,N^{\prime\prime})\}.$$
  \item[\rm{(}d\rm{)}] ${\cd}_{(\fa,\mathcal S)}(M,R)=\sup\{{\cd}_{(\fa,\mathcal S)}(M,N)|N
    \text{ is a finite } R\text{--module } \}$.
  \item[\rm{(}e\rm{)}] ${\cd}_{(\fa,\mathcal S)}(M,N)=
    \sup\{{\cd}_{(\fa,\mathcal S)}(M,R/\fp)|\fp\in\Supp_R(N) \}$.
  \item[\rm{(}f\rm{)}] ${\cd}_{(\fa,\mathcal S)}(M,N)=
    \sup\{{\cd}_{(\fa,\mathcal S)}(M,R/\fp)|\fp\in\Min\Ass_R(N) \}$.
\end{enumerate}
If $\mathcal M$ is  Melkersson  subcategory, then the following
statements hold:
\begin{enumerate}
  \item[\rm{(}g\rm{)}] ${\cd}_{(\fa,\mathcal M)}(M,N)=
    \min\{r\geq 0|\lc^{r+1}_{\fa}(M,R/\fp)\in\mathcal M
     \text{ for all }\fp\in\Supp_R(N) \}$.
  \item[\rm{(}h\rm{)}] For each integer $i$ with $1+\pd M\leq i\leq{\cd}_{(\fa,\mathcal M)}(M,N)+\pd M$,
     there exists $\fp\in \Supp_R(N)$ with
     $\lc^{i}_{\fa}(M,R/\fp)$ not in $\mathcal M$.
  \item[\rm{(}i\rm{)}] ${\cd}_{(\fa,\mathcal M)}(M,R)=
    \min\{r\geq 0|\lc^{r+1}_{\fa}(M,R/\fp)\in \mathcal M
\text{ for all }\fp\in\Spec(R)\}$.
 \item[\rm{(}j\rm{)}] ${\cd}_{(\fa,\mathcal M)}(M,R)=
    \min\{r\geq 0|\lc^{r+1}_{\fa}(M,N)\in \mathcal M
\text{  for all finite }$$R\text{--modules } N \}$.
\end{enumerate}

\end{prop}
\begin{proof}
(a) By definition.

(b) Follows from \ref{2-1}.

(c) The inequality "$\geq$", holds by (b) and we get the opposite
inequality from the following exact sequence
$$
\dots\lo \lc^{i}_{\fa}(M,N^{\prime})\lo \lc^{i}_{\fa}(M,N) \lo
\lc^{i}_{\fa}(M,N^{\prime\prime})\lo \dots
$$

The assertions (d), (e) and (f) follow from theorem \ref{2-1}
 $(i)\Leftrightarrow{(ii)}$, $(i)\Leftrightarrow{(ii)}$
and ${(i)}\Leftrightarrow{(iv)}$, respectively.

(g), (h), (i) and (j) follow from \ref{T:casen+1}.
\end{proof}

\bibliographystyle{amsplain}

\begin{thebibliography}{9}


\bibitem{AM}
M. Aghapournahr, L. Melkersson, \emph{Local cohomology and Serre
subcategories}, J. Algebra., \textbf{320} (2008), 1275--1287.


\bibitem{ATV}
M. Aghapournahr, A. J. Taherizadeh, A. Vahidi, \emph{
Extension functors of local cohomology modules}, To appear in IBM.

\bibitem{AN}
J. Amjadi, R. Naghipour, \emph{Cohomological dimension of local cohomology modules},
  Alg. Colloq, \textbf{320} (2008), 1275--1287.

\bibitem{AT}
M. Asgharzadeh, M. Tousi, {\it A unified approach to local
cohomology modules using Serre classes}, Canad. Math. Bull., 53 (2010), 577-586.


\bibitem{ADT}
M. Asgharzadeh, K. Divanni-Aazar, M. Tousi, \emph{The finiteness dimension
of local cohomology modules and its dual notion},
  Pure. Appl. Algebra, \textbf{320} (2008), 1275--1287.

\bibitem{BZ} M. H. Bijan-Zadeh, {\it A commen generalization of local cohomology theories},
Glasgow Math. J.  {\bf 21}(1980), 173-181.

\bibitem{Bo} N. Bourbaki, {\it Alg$\grave{e}$bre commutative}, Chap.1-Chap.9. Hermann, 1961-83.

\bibitem{BSh}
M .P. Brodmann, R. Y. Sharp, {\it Local cohomology : an algebraic
introduction with geometric applications}, Cambridge University
Press, 1998.

\bibitem{BH}
W. Bruns, J. Herzog, {\it Cohen-Macaulay rings}, Cambridge
University Press, revised ed., 1998.





\bibitem{CT}
 L. Chu, Z. Tang, \emph{On the artinianness of  generalized local cohomology
}, Comm. Algebra.,
\textbf{35} (2007), 3821--3827.

\bibitem{CH}
  N. T. Coung, N. V. Hoang,
   \emph{On the vanishing and the finiteness of supports of generalized local cohomology modules
   },
   manuscripta math. \textbf{104}(2001), 519--525.

\bibitem{DY} M.T . Dibaei, S. Yassemi, {\it Associated primes
and cofiniteness of local cohomology modules}, manuscripta math, {\bf 117}(2005), 199-205.

\bibitem{DNT} K. Divaani-Aazar, R. Naghipour, M. Tousi, {\it Cohomological dimension of certain algebraic varieties},
 Proc. Amer. Math. Soc. {\bf 130}(2002), 3537-3544.

\bibitem{F} G. Faltings, {\it \"{U}ber lokale kohomologiegruppen hoher Ordnung},
J. Reine Angew. Math. {\bf 313}(1980), 43-51.


\bibitem{Ha}
R. Hartshorne, {\it Cohomological dimension of algeraic varieties},
Ann. of Math. \textbf{88} (1968), 403--450.

\bibitem{He}
    J. Herzog, \emph{Komplexe, Aufl\"{o}sungen und Dualit\"{a}t in der lokalen Algebra},
    Habilitationsschrift, Universitat Regensburg 1970.
    Invent. Math. \textbf{9} (1970), 145--164.


\bibitem{HL}
C. Huneke, G. Lyubeznik, \emph{On the vanishing of local cohomology modules},
Inv. Math. \textbf{102} (1990), 73--93.



\bibitem{Mel} L. Melkersson, {\it Modules cofinite with respect to an
ideal}, J. Algebra. {\bf 285}(2005), 649-668.


\bibitem{O} A. Ogus, {\it Local cohomological dimension of algebraic varieties},
 Annala. of Math. {\bf 98}(1973), 327-396.


\bibitem{S} N. Suzuki, \emph{On the generalized local cohomology and its duality}, J.
    Math. Kyoto. Univ.
    \textbf{18} (1978), 71--85.


\bibitem{V} W. Vasconcelos, {\it Divisor theory in module
categories}, North- Holland, Amsterdam, 1974.

\bibitem{Ya}
S. Yassemi, \emph{Generalized section functors}, J. Pure. Apple. Algebra
 \textbf{95} (1994), 103-119.


\end{thebibliography}

\end{document}